\documentclass[twoside]{article}
\usepackage{a4,times,amsmath,theorem,latexsym,amssymb}

\advance\hoffset by -0.3cm

\makeatletter
\def\@maketitle{%
  \newpage
  {\vspace*{-11ex}%
  \small\noindent}
  \null
  \vskip 5em%
  \begin{flushleft}%
    {\LARGE \bf\@title \par}
    \vskip 1.5em%
       {\large\bf \lineskip .5em \@author \par}%
    \vskip 1.0em%
  \end{flushleft}%
  \par
  \vskip 1.5em}
\renewcommand\section{\@startsection {section}{1}{\z@}%
                                   {-3.5ex \@plus -1ex \@minus -.2ex}%
                                   {1.5ex \@plus.2ex}%
                                   {\normalfont\Large\bfseries}}
\renewcommand\subsection{\@startsection{subsection}{2}{\z@}%
                                     {-3.0ex\@plus -2ex \@minus -.2ex}%
                                     {1.0ex \@plus .2ex}%
                                     {\normalfont\large\bfseries}}
\renewcommand\subsubsection{\@startsection{subsubsection}{3}{\z@}%
                                     {-2.5ex\@plus -2ex \@minus -.2ex}%
                                     {0.9ex \@plus .2ex}%
                                     {\normalfont\normalsize\bfseries}}
\renewcommand\paragraph{\@startsection{paragraph}{4}{\z@}%
                                     {-2.0ex \@plus1ex \@minus.2ex}%
                                     {-1em}%
                                     {\normalfont\normalsize\bfseries}}
\renewcommand\subparagraph{\@startsection{subparagraph}{5}{\parindent}%
                                       {-1.5ex \@plus1ex \@minus .2ex}%
                                       {-1em}%
                                      {\normalfont\normalsize\bfseries}}
\def\@evenhead{\footnotesize\thepage\hfil\slshape\leftmark}%
\def\@oddhead{\footnotesize{\slshape\rightmark}\hfil\thepage}%
\numberwithin{equation}{section} \numberwithin{figure}{section}
\numberwithin{table}{section}

\renewcommand{\@makecaption}[2]{\begin{quote}
\footnotesize {\bf #1}~#2
\end{quote}}
\makeatother
\newenvironment{summary}{\vskip\baselineskip \noindent\small\bf Summary\\ \rm}%
{\vskip\baselineskip}
\newenvironment{proof}{{\vskip\baselineskip\noindent\textbf{Proof:}}}%
{\hspace*{.1pt}\hspace*{\fill}\BOX\vskip\baselineskip}

\newcommand{\BOX}{\ensuremath\Box}
\newtheorem{theorem}{Theorem }[section]

\newtheorem{corollary}[theorem]{Corollary}
\newtheorem{definition}[theorem]{Definition}
{\theorembodyfont{\rmfamily}}
{\theorembodyfont{\rmfamily}\newtheorem{example}[theorem]{Example}}
{\theorembodyfont{\rmfamily}}
\newtheorem{lemma}[theorem]{Lemma}
\newtheorem{proposition}[theorem]{Proposition}
{\theorembodyfont{\rmfamily}\newtheorem{remark}[theorem]{Remark}}
{\theorembodyfont{\rmfamily}}

\title{Statistical likelihood methods in finance}
\author{Arnold Janssen${}^{*}$, Martin Tietje}

\newcommand{\ind}{1\!\! 1}

\begin{document}
\maketitle\thispagestyle{empty}

Heinrich-Heine-Universit\"at, Universit\"atsstr. 1, 40225 D\"usseldorf, Germany\\
$~~~~~~\ $Emails: janssena@math.uni-duesseldorf.de, tietje@math.uni-duesseldorf.de

\begin{summary}
It is known from previous work of the authors that non-negative arbitrage free price processes in finance can be described in terms of filtered likelihood 
processes of statistical experiments and vice versa. The present paper summarizes and outlines some similarities between finance and the statistical likelihood 
theory of Le Cam. Options are linked to statistical tests of the underlying experiments. In particular, some price formulas for options are expressed by the 
power of related tests. In special cases the dynamics of power functions for filtered likelihood processes can be used to establish trading strategies which lead 
to formulas for the Greeks $\Delta$ and $\Gamma$. Moreover statistical arguments are then used to establish a discrete approximation of continuous time trading 
strategies. It is explained that It\^{o} type financial models correspond to hazard based survival models in statistics. Also price processes given by a 
geometric fractional Brownian motion have a statistical counterpart in terms of the likelihood theory of Gaussian statistical experiments.
\end{summary} 
 
\renewcommand{\thefootnote}{}
\footnotetext{\hspace*{-.51cm}AMS 1991 subject classification: Primary: 91B02; Secondary: 62B15, 91B24\\ 
Key words and phrases: statistical experiment, filtered likelihood ratio process, pricing formula, hazard rate, hedging strategy, fractional Brownian motion, 
local asymptotic Wiener\\
*) corresponding author}

\section{Introduction}\label{section1}
In this paper we consider a positive arbitrage free financial model as in Janssen and Tietje (2013) using their notation. To be more precise, let 
$(X_t^i)_{t \in [0,T]}$, $1 \le i \le d$ be $d$ adapted positive discounted price processes with filtration $(\mathcal{F}_t)_{t \in [0,T]}$, 
finite time horizon $T$ and real world measure $P$. By Janssen and Tietje we have the following useful equivalence:
\begin{proposition}\label{equiv}
Let $Q$ be a probability measure equivalent to $P$. The following assertions are equivalent:
\begin{enumerate}
   \item[(1)] There are probability measures $Q_1, ..., Q_d$ satisfying
 \begin{equation}
    \frac{dQ_{i|\mathcal{F}_t}}{dQ_{|\mathcal{F}_t}} = \frac{X^i_t}{X^i_0}, \ \ \ t \in [0,T], 
\textit{ and}~ Q_i \ll Q \textit{ for all}~ 1 \le i \le d.
 \end{equation}
   \item[(2)] $Q$ is a martingale measure, i.e. $(X_t^i)_{t \in [0,T]}$ is a $Q$-martingale for each $1 \le i \le d$.
\end{enumerate}
\end{proposition}
Whenever \ref{equiv} holds $\{Q_1,...,Q_d, Q\}$ can be viewed as a stastical experiment which is then called financial 
experiment where the physical measure $P$ is often suppressed. 
Thus, positive arbitrage free financial asset models can be translated in terms of Le Cam's theory of 
statistical experiments. They correspond one to one to so called filtered financial experiments given by 
likelihood processes. Since likelihood processes are well studied in statistics various similarities between finance 
and statistics could be derived, in particular for option prices and power functions of statistical tests, 
completeness of financial markets and complete statistical experiment among different other topics. Also discrete approximations 
of price processes and option prices were linked to the convergence of statistical experiments and power functions of tests.\\
In this paper we will present further consequences and results of that approach. Janssen and Tietje (2013) offered an
alternative statistical point of view for price formulas for various payoff options $H$. Consider for instance a European 
call option $H_C$ with strike price $K$ with a bond given by $S_t^0=\exp\left(\int_0^t \rho(u) du\right)$ and a deterministic interest rate 
$\rho:[0,T] \rightarrow \mathbb{R}$. If $S_t^i$ denotes the $i$-th asset then $X_t^i:=\frac{S_t^i}{S_t^0}$ holds. By Janssen and Tietje, Example 5, the option 
price at time $t$ with initial value $s^1_t$ of the asset can be represented by
\begin{equation}\label{price0}
 p_Q(H_C,s_t^1,t)=s_t^1 E_{Q'_1(t)}\left( \Phi_t\left( \frac{dQ'_1(t)}{dQ},s_t^1\right) \right)
\end{equation}
$~~~~~~~~~~~~~~~~~~~~~~~~~~~~~~~~~~~~~~~~~~~~~~~~~~~~~~~~~~~-\exp\left( -\int_t^T \rho(u)du\right)K 
 E_Q\left( \Phi_t\left( \frac{dQ'_1(t)}{dQ},s_t^1\right) \right)$
where $E_{Q_1'}(\Phi_t)$ and $E_Q(\Phi_t)$ are the power of suitable tests $\Phi_t$ for the corresponding 
experiments at time $t$, see \eqref{hct} - \eqref{delta} below for more information. For continuous time models it can be shown that the decomposition of the 
price \eqref{price0} allows a nice interpretation of the ``Greeks $\Delta$ and $\Gamma$''. Under some regularity assumptions the 
$\Delta$ is just the power of $\Phi_t$ under the alternative $Q_1'(t)$
\begin{equation}
 E_{Q_1'(t)}(\Phi_t) = \xi_1(t) \quad \textnormal{(`` } \Delta \textnormal{ Delta'')}
\end{equation}
which coincides with the hedging strategy $\xi_1(t)$ of the required portion of the asset at time $t$. That
result can be used to establish discrete approximations of continuous time hedging strategies. The result is based on statistical
 arguments. The convergence of likelihood processes implies the convergence of accompanying power functions of tests. In particular, 
 each of the power terms of \eqref{price0} are convergent in various cases.\\
 
 Section \ref{section3} outlines the connection between It\^{o} type financial models and hazard based survival models. The 
 volatility corresponds to hazard rate derivatives which are used in survival analysis. Hazards are time dependent failure 
 rates which serve as main parameters in health science and insurance. This connection opens the door for a comparison of 
 financial models and hazard based models in medicine, see Andersen et al. (1993). Here we offer a functional limit theorem 
 for positive price processes (likelihood processes). The statistical local asymptotic Wiener (LAW) property is here of importance. 
 Finally, some auxiliary material about local asymptotic mixed normal models (LAMN) in finance are presented. Also experiments with 
 a geometric fractional Brownian motion part are known in statistics. Thus, the likelihood theory is also of interest beyond arbitrage free models.\\

\section{Option prices, hedging strategies and the power of statistical tests}\label{section2}

Suppose that always a martingale measure $Q$ exists. Consider the payoff $H$ at time $T$
\begin{equation}\label{h}
 H:= (S_T^1 - K) \Phi (\mathbb{S}), \quad 0 \le \Phi \le 1,
\end{equation}
with 
$\mathbb{S}=\left( \frac{dQ_{i|\mathcal{F}_t}}{dQ_{|\mathcal{F}_t}}\right)_{t \le T}
=\left( \frac{S_t^i}{s_0^i \exp\left( \int_0^t \rho(u)du\right) }\right)_{t \le T}$
as in Janssen and Tietje (2013). They introduced a testing problem for the null hypothesis $\{Q\}$ versus the alternative $\{Q_1\}$ where 
$\Phi = \Phi(\mathbb{S})$ serves as a statistical test. The $Q$-price of $H$ at time $t=0$ is defined by
\begin{equation}\label{price1}
 p_Q(H):= E_Q((S_T^0)^{-1}H).
\end{equation}
Note that at present nothing is said about the uniqueness of the martingale measure $Q$. However, under additional assumptions like the completeness of the 
market the value \eqref{price1} yields the unique option price of $H$, see Karatzas and Shreve (1991), p. 378.

\begin{lemma}[Janssen and Tietje (2013), Prop. 2]
${}$\\
 The $Q$-price of $H$ is a linear combination of the $Q$-level $E_Q(\Phi)$ of the test $\Phi=\Phi(\mathbb{S})$ and its power $E_{Q_1}(\Phi)$ namely
 \begin{equation}\label{pqh}
 p_Q(H)=s_0^1 E_{Q_1}\left( \Phi \right)-\exp\left( -\int_0^T \rho(u)du\right)K E_Q\left( \Phi \right).
\end{equation}
\end{lemma}

In comparison with \eqref{price1} the power formula \eqref{pqh} may have some computational advantages.

\begin{remark}[Applications of the statistical approach of the price \eqref{pqh}]
 ${}$\\Suppose that $H$ is a complicated option where no explicit price is known. Then $p_Q(H)$ can be calculated by Monte Carlo experiments via \eqref{price1} 
 or \eqref{pqh}. Observe that the 
 $Q$-variance of the typically unbounded option $H$ can be large at least for a large time horizon $T$. That variance decreases the quality of the Monte Carlo 
 approach for the right hand side of \eqref{price1}. We suggest to carry out two Monte Carlo experiments for $E_{Q_1}(\Phi)$ and $E_Q(\Phi)$ in \eqref{pqh}. It is 
 our experience that the accuracy of the Monte Carlo approximation is better here as in \eqref{price1} for larger $Var_Q(H)$ since $0 \le \Phi \le 1$ controls 
 the variance in \eqref{pqh}.
\end{remark}

Next we will explain the meaning and the consequences of the power decomposition of $p_Q(H)$ in \eqref{pqh}. For these reasons we refer to the dynamics of the 
option price as function of the time $t$, see Janssen and Tietje, Example 5. For simplicity we consider the European call
\begin{equation}
 H_C = (S_T^1 - K) \Phi_C \left(\frac{dQ_1}{dQ} \right)
\end{equation}
with $\Phi_C(x) = \ind_{\{x > c\}}$ and $c = \frac K{s_0^1} exp\left( - \int_0^t \rho(u) du\right)$. As stated in Janssen and Tietje there exists another 
measure $Q_1'(t) \ll Q$ such that the updated price at time $0 \le t \le T$ 
\begin{equation}
 \frac{X^i_T}{X^i_t}= \frac{dQ'_i(t)}{dQ} \textnormal{ with } \frac{dQ_1}{dQ} = \frac{dQ_{1|\mathcal{F}_t}}{dQ_{|\mathcal{F}_t}} \frac{dQ'_1(t)}{dQ}
\end{equation}
is a likelihood ratio. The updated option $H$ of \eqref{h} given $S_t^1=s_t^1$ has now the form
\begin{equation}\label{hct}
 H_{C,t}=(S_T^1 - K) \Phi_t \left(\frac{X_T^1}{X_t^1}, s_t^1 \right)
\end{equation}
with 
\begin{eqnarray*}
 \Phi_t\left( \frac{X_T^1}{X_t^1},s_t^1\right)
&=&\Phi_C\left(  \frac{s_t^1}{s_0^1 \exp\left( \int_0^t \rho(u)du\right) } \cdot\frac{dQ'_1(t)}{dQ} \right)\\
&=&\ind\left( \frac{dQ'_1(t)}{dQ} > \frac K{s_t^1} exp\left( - \int_t^T \rho(u) du \right) \right).
\end{eqnarray*}

\begin{theorem}\label{call}
 (a) The $Q$-price process at time $t \le T$ of the European call $H_C$ with initial value $S_t^1 = s_t^1$ is given by
 \begin{equation}\label{pqhct}
 p_Q(H_C,s_t^1,t)=s_t^1 E_{Q'_1(t)}\left( \Phi_t\left( \frac{dQ'_1(t)}{dQ},s_t^1\right) \right)
\end{equation}
$~~~~~~~~~~~~~~~~~~~~~~~~~~~~~~~~~~~~~~~~~~~~~~~~~~~~~~~~~~~-\exp\left( -\int_t^T \rho(u)du\right)K 
 E_Q\left( \Phi_t\left( \frac{dQ'_1(t)}{dQ},s_t^1\right) \right).$
\end{theorem}
(b) Suppose that for all $t$ the likelihood distribution $\mathcal{L}\left(\frac{dQ'_1(t)}{dQ}\Big|Q\right)$ has a continuous Lebesgue density on $(0,\infty)$. 
Then the ``delta''($\Delta$) of the price process 
\begin{equation}\label{delta}
 \frac d{dx} p_Q(H_C,x,t)_{|x=s_t^1} = E_{Q_1'(t)}\left( \Phi_t\left( \frac{dQ'_1(t)}{dQ},s_t^1\right) \right)
\end{equation}
is just the $Q'_1(t)$ power of the test $\Phi_t$.

\begin{remark}
 (a) Under additional assumptions, for instance for complete financial models with It\^{o} type price processes, the delta \eqref{delta} gives just the 
 investment part $\xi_1(t)s_t^1$ at time $t$ of the unique self financing hedging strategy $(\xi_0(t), \xi_1(t))$, see F\"ollmer and Schied (2004). Under this 
 condition the hedging strategy for $H_C$ is completely given by the power functions of our tests
 \begin{equation}\label{hedge}
  \xi_1(t) = E_{Q_1'(t)}\left( \Phi_t\left( \frac{dQ'_1(t)}{dQ},s_t^1\right) \right).
 \end{equation}
Thus, the statistical quantities of the price decomposition have a concrete meaning for the hedging strategy.\\
(b) Of course \eqref{hedge} is well known for the famous Black-Scholes price for the European call.\\
(c) It turned out that formula \eqref{hedge} does not hold for other options in general or the price of $H_C$ under a discrete model like the 
Cox-Ross-Rubinstein model.
\end{remark}
${}$\\
\textbf{Proof of Theorem \ref{call}:}
Part (a) is proved in Janssen and Tietje.\\
(b) Let $f_0: (0,\infty) \rightarrow [0,\infty)$ denote the continuous density of $\mu_1 = \mathcal{L}\left(\frac{dQ'_1(t)}{dQ}\Big|Q\right)$. Since 
$Q_1'(t) \ll Q$ holds we have $\frac{d\mu_2}{d\mu_1}(y)=y$ for $\mu_2 = \mathcal{L}\left(\frac{dQ'_1(t)}{dQ}\Big|Q_1'(t)\right)$ which is frequently used in 
statistics and Le Cam's theory of experiments. Thus, $\mu_2$ has the density $y \mapsto y f_0(y)$. If we put $c_t = K exp \left( - \int_t^T \rho(u) du \right)$ 
the price formula \eqref{pqhct} reads as
\begin{equation*}
 p_Q(H_C,s_t^1,t)=s_t^1 \int \ind \left( y > \frac{c_t}{s_t^1} \right) d\mu_2
\end{equation*}
$~~~~~~~~~~~~~~~~~~~~~~~~~~~~~~~~~~~~~~~~~~~~~~~~~~~~~~~~~~~~~~~~~~~~~-\exp\left( -\int_t^T \rho(u)du\right)K 
 \int \ind \left( y > \frac{c_t}{s_t^1} \right) d\mu_1$
 and
 \begin{align*}
 &\frac d{dx} p_Q(H_C,x,t) - E_{Q_1'(t)}\left( \Phi_t\left( \frac{dQ'_1(t)}{dQ},x\right) \right)\\
 &=x \frac d{dx} \int_{\frac{c_t}x}^{\infty} y f_0(y)dy - exp\left( -\int_t^T \rho(u)du\right)K \frac d{dx}\int_{\frac{c_t}x}^{\infty} f_0(y)dy\\
 &=\left[ x \frac{c_t}x f_0\left(\frac{c_t}x \right)-c_t f_0\left(\frac{c_t}x\right)\right] \frac{c_t}{x^2} = 0.
 \end{align*}
\hspace*{\fill}\begin{small}$\square$\end{small}

\begin{remark}[Power identity]
${}$\\
 (a) In statistical terms the test $\Phi_t(\cdot,\cdot)$ is a Neyman Pearson test at level $E_Q\left( \Phi_t\left( \frac{dQ'_1(t)}{dQ},s_t^1\right) \right)$ for 
 the null hypothesis $\{Q\}$. The following well known identity, see for instance Krafft and Plachky (1970), has now an interpretation for the hedging strategy of 
 the European call. When \eqref{hedge} holds we have
 \begin{align*}
  \xi_1(t) &= E_{Q_1'(t)}\left( \Phi_t\left( \frac{dQ'_1(t)}{dQ},s_t^1\right) \right)\\
  &= inf_{k \ge 0} \left\{ k E_Q\left( \Phi_t\left( \frac{X_T^1}{X_t^1},s_t^1\right) \right) + 
  \int \left( \frac{X_T^1}{X_t^1}-k\right)^+ dQ \right\},
 \end{align*}
where $\dfrac{dQ_1'(t)}{dQ} = \dfrac{X_T^1}{X_t^1}$.\\
 (b) Under the present assumptions we have the power relation for the ``Greek'' $\Gamma$ 
 \begin{align*}
  \Gamma(x,t) &= \frac{d^2}{dx^2} p_Q(H_c,x,t) = \frac d{dx} E_{Q_1'(t)}\left( \Phi_t\left( \frac{dQ'_1(t)}{dQ},x\right) \right)\\
  &=\frac {K exp\left(-\int_t^T \rho(u)du\right)} x \frac d{dx} E_Q\left( \Phi_t\left( \frac{dQ'_1(t)}{dQ},x\right) \right).
 \end{align*}
\end{remark}

The power formula \eqref{hedge} enables us to establish an approximation or discretization of the underlying hedging strategy for the European call. This is due 
to the fact that the convergence of statistical experiments implies the convergence of the power of Neyman Pearson tests.

\begin{remark}[Convergence of hedging strategies in terms of power functions]
${}$\\
Suppose that the financial model of Theorem \ref{call} is complete with a unique hedging strategy $\xi_1(t)$ given by the power \eqref{hedge} of the European 
call.\\
(a) Often $\dfrac{X_T^1}{X_t^1}=\dfrac{dQ_1(t)}{dQ}$ can be approximated by discrete financial experiments, see Janssen and Tietje, sections 5, 6. As concrete 
example they considered a proper parametrization of the Cox-Ross-Rubinstein model. Thus, the related Neyman Pearson power converges to the unknown hedging 
strategy \eqref{hedge}.\\
(b) However, the Neyman Pearson power for the Cox-Ross-Rubinstein model may be no longer the hedging strategy in general for discrete models since the 
assumptions of Theorem \ref{call} are not fulfilled.
\end{remark}

\section{Volatility and hazard parameters in statistics}\label{section3}

In this section the connection between statistical survival models and It\^{o} type financial models is pointed out. In particular, we explain that the
volatility is connected to hazard rates and survival models which are well studied in statistics. The subsequent approach does not cover the most general
case. For convenience the connection is explained for deterministic volatilities in order to present the main idea. Of course more advanced survival models 
allow predictable hazard rates (volatilities), see Andersen, Borgan, Gill and Keiding (1993).\\
(I) Models given by independent components.\\
Throughout let $P_0$ be the uniform distribution on the unit interval. As statistical parameter space $\Theta$ we choose the set of measurable functions
$$\Theta = \left\{g:[0,1] \rightarrow \mathbb{R},~ g \textnormal{ bounded}, \int g^2 dP_0 < \infty, \int g dP_0 = 0\right\}.$$
For small enough $\vartheta \in \mathbb{R}$ each so called tangent $g \in \Theta$ defines a path of distributions via
\begin{equation}\label{path}
 \frac{dP_{\vartheta}}{dP_0}= 1 + \vartheta g.
\end{equation}
Consider now a finite number of tangents $g_1,...,g_d$. Let
\begin{equation}
 \Sigma:=(\sigma_{ij})_{i,j \le d}:=(\textnormal{Cov}_{P_0}(g_i,g_j))_{i,j \le d}
\end{equation}
denote the covariance matrix which is assumed to have full rank. For large $n$ we can now define a statistical experiment
$E_n=\{Q_{1,n},...,Q_{d,n},Q_{0,n}\}$ on $[0,1]^n$ by
\begin{equation}\label{likelihood1}
 \frac{dQ_{i,n}}{dQ_{0,n}}(x)=\prod_{j=1}^n \left(1 + \frac{g_i(x_j)}{\sqrt n}\right), \quad x=(x_1,...,x_n) \in [0,1]^n
\end{equation}
where $Q_{0,n} = P_0^n$.
Our time interval $[0,T]=[0,1]$ is now the unit interval which introduces the canonical filtration on $[0,1]^n$
\begin{equation}
 \mathcal{F}_{n,t} = \sigma ( \pi_s^i: s \le t, i \le n), \quad t \in [0,1],
\end{equation}
by the indicators $\pi_s^i(x) = \ind_{[0,s]}(x_i)$ of the canonical projection on $x_i$.\\
By Prop. \ref{equiv} the model defines price processes $\mathbb{X}_n=((X_{n,t}^i)_{t \in [0,1], i \le d})$
\begin{equation}\label{price}
 X_{n,t}^i:=\frac{dQ_{i,n |\mathcal{F}_{n,t}}}{dQ_{0,n |\mathcal{F}_{n,t}}}, \quad t\in [0,1], 1 \le i \le n.
\end{equation}
The asymptotics of this model will give more insight into the interaction of statistical models and It\^{o} type financial models. We begin with the celebrated 
well-known local asymptotic normality of Le Cam which is actually the central limit theorem for the sequence $(E_n)_n$ of statistical experiments.
For each $g_i$ we have the LAN expansion
\begin{equation}\label{lan}
 log \frac{dQ_{i,n}}{dQ_{0,n}}(x) - \left( Z_n(g_i)(x) - \frac{\sigma_i^2}2 \right) \longrightarrow 0
\end{equation}
in $Q_{0,n}$ probability where $Z_n(g_i)(x):= \frac 1{\sqrt n} \sum_{j=1}^n g_i(x_j)$ is the central sequence with $Q_{0,n}$ variance 
$\sigma_i^2 = \int g_i^2 dP_0$.\\
The result is based on the Taylor expansion $log(1+x) \approx x - \frac {x^2}2$. The limit model of $E_n$ is $E = (Q_1,...,,Q_d,Q)$ where $Q = N(0, \Sigma)$ 
and $Q_i = N(\Sigma e_i, \Sigma)$, $e_i$ the $i$-th unit vector on $\mathbb{R}^d$ with
\begin{equation}\label{dens1}
 \frac{dQ_i}{dQ}(x) = exp\left(x_i - \frac{\sigma_i^2}2 \right).
\end{equation}
Below we use the weak convergence of sequences of statistical experiments, see for instance Le Cam and Yang (2000) or Janssen and Tietje for the connection to 
financial experiments.
\begin{proposition}\label{conv1}
 We have weak convergence of the experiments $E_n \rightarrow E$, i.e. weak convergence of the distributions
 $$\mathcal{L}\left( \left( log \frac{dQ_{i,n}}{dQ_{0,n}} \right)_{i \le d} \Big| Q_{0,n}\right)
  \longrightarrow \mathcal{L}\left( \left( log \frac{dQ_i}{dQ} \right)_{i \le d} \Big| Q\right)
  $$
on $\mathbb{R}^d$.
\end{proposition}

\begin{proof}
 Observe that $\mathcal{L}\left( \left( Z_n(g_i) \right)_{i \le d} | Q_{0,n}\right)$ is asymptotically $N(0,\Sigma)$ by the central limit theorem.
 Thus \eqref{dens1}, \eqref{conv1} and the Cram\'{e}r Wold device proves the result.
\end{proof}

Until now nothing is said about the influence of the filtration. It contains the sequential aspect which is typically modeled by hazards and survival 
aspects of the models when for instance life time data show up sequentially. To motivate this we go back to the path of distributions $P_{\vartheta}$ given in 
\eqref{path} where the tangent $g$ is given by $\frac d{d\vartheta} \frac{dP_{\vartheta}}{dP_0}(x)_{|\vartheta =0} = g(x)$. On the other hand the distributions 
$P_{\vartheta}$ can also be described by their hazard rates
\begin{equation}
 \lambda_{\vartheta}(x):= \frac{1+\vartheta g(x)}{P_{\vartheta}([x,1])}, \quad x \in [0,1],
\end{equation}
which is a time dependent failure rate.
It is well known that
\begin{equation} 
 \frac d{d\vartheta} \frac{\lambda_{\vartheta}(x)}{\lambda_0(x)}_{|\vartheta=0} = g(x) - \frac{\int_x^1 g dP_0}{1-x} = : R(g)(x)
\end{equation}
holds as stochastic derivative, see Efron and Johnstone (1990), Ritov and Wellner (1988) and Janssen (1994).
The function $R(g)$ is called the hazard rate derivative at $\vartheta = 0$ of \eqref{path}.
That operator $R$ is an isometry
\begin{equation}
 R : L_2^0(P_0) := \{ g \in L_2(g): \int g dP_0 = 0 \} \rightarrow L_2(P_0)
\end{equation}
w.r.t. the inner product of $L_2(P_0)$. Its inverse operator is
\begin{equation}
 L(\gamma)(x) = \gamma(x) - \int_0^x \frac{\gamma(u)}{1-u} du \textnormal{ for } \gamma \in L_2(P_0).
\end{equation}
Moreover, it is easy to see that for $\gamma = R(g)$
\begin{equation}\label{l}
 L(\gamma \ind_{[0,t]}) = E_{P_0}(g | \mathcal{F}_{1,t})
\end{equation}
holds which is a key observation. For these reasons the statistical models are now reparametrized by the hazard rate derivatives 
$\gamma = R(g) \in R(\Theta)$.\\
Accordingly, we have for $\gamma_i:=R(g_i)$ that the price process
\begin{equation}\label{pp1}
 X_{n,t}^i(x) = \frac{dQ_{i,n |\mathcal{F}_{n,t}}}{dQ_{0,n |\mathcal{F}_{n,t}}}(x) = \prod_{j=1}^n \left(1+\frac{L(\gamma_i \ind_{[0,t]})(x_j)}{\sqrt{n}}\right)
\end{equation}
is a filtered likelihood process given by the hazard quantities $\gamma_i$. The present right hand side of \eqref{pp1} is now used to establish the asymptotics 
of the price process $X_{n,t}^i$ for fixed $t$ via the appertaining statistical experiments 
$E_n(t):=(Q_{1,n|\mathcal{F}_{n,t}},...,Q_{d,n|\mathcal{F}_{n,t}},Q_{0,n})$ of \eqref{likelihood1}.

\begin{corollary}\label{conv2}
 For fixed $t$ we have weak convergence of the experiments $E_n(t) \rightarrow E(t)$ in the sense of Prop. \ref{conv1} where $E(t)=(Q_1(t),...,Q_d(t),Q(t))$ with
 \begin{equation}
  Q_i(t)=N(\Sigma(t)e_i,\Sigma(t)), \quad \Sigma(t) =( \sigma_{ij}(t) ) = \left(\int_0^t \gamma_i(u)\gamma_j(u)du\right).
 \end{equation}
\end{corollary}
The proof follows from Proposition \ref{conv1}. The covariance matrix is given by $E_{P_0}(g_i|\mathcal{F}_{1,t})$ which can be expressed by the hazard 
quantities by \eqref{l} and the isometry $R$, $L$ respectively.\\
\\ 
Observe that at $t=1$ we have $\Sigma(1) = \Sigma$ and here the tangent and hazard approach are the same. For $t<1$ our hazard rate approach leads to pointwise 
limit distributions of the underlying price process. In the next step their process structure is studied and we will connect the price processes with filtered 
likelihood processes. To motivate this let $t \mapsto \hat{F}_n(t)$ denote the empirical process of $n$ i.i.d. uniformly distributed random variables. Under 
$Q_{i,n}$ the corresponding normalized empirical process
\begin{equation}\label{empirical}
 \sqrt{n} \left(\hat{F}_n(t)+\frac 1{\sqrt n}\int_0^t g_i(u) du - t\right) \rightarrow B_0(t) + \int_0^t g_i(u) du =: Y_t
\end{equation}
converge weakly on $D[0,1]$ to a shifted standard Brownian bridge $B_0$ where \\
$t \mapsto \int_0^t g_i(u) du$ is an unknown signal. Consider for a moment a model with one asset, i.e. $d=1$. Define the distribution
\begin{equation}
 \nu_{g_1}:=\mathcal{L}((Y_t)_{t \in [0,1]}|g_1) \textnormal{ on } C[0,1]
\end{equation}
under the parameter $g_1$.
 
\begin{corollary}
 The limit experiments $E=(Q_1,Q)$ and $E'=(\nu_{g_1},\nu_0)$ are equivalent.
\end{corollary}

The proof is based on the Girsanov formula
\begin{equation}
 \frac{d\nu_{g_1}}{d\nu_0}=exp\left( \int_0^1 g_1 dB_0 - \frac{\int_0^1 g_1(u)^2 du}2\right).
\end{equation}
Now it is easy to see that the likelihood distributions of $E$ and $E'$ coincide.\\
\\
We see that at time $t=1$ the limit experiment is given by the signal detection model with noise part $B_0$ which is motivated by  \eqref{empirical}. However, 
$B_0$ is not appropriate for the sequential financial setup. Next we will see that the hazard reparametrization by hazard quantities works well for the sequential 
approach. Recall that the Doob Meyer decomposition of $B_0$ w.r.t. $\mathcal{F}_t = \sigma(B_0(s) : s \le t)$ is given by
\begin{equation}\label{b0}
 B_0(t) = B(t) - \int_0^t \frac{B_0(s)}{1-s} ds,
\end{equation}
where $B(t)$ is a standard Brownian motion. For each $\gamma \in L_2(P_0)$ let
\begin{equation}\label{signal}
 \xi_t = B(t) + \int_0^t \gamma(u) du, \quad t \in [0,1],
\end{equation} 
denote the signal detection model with noise part $B(t)$. The Doob Meyer decomposition of the empirical process was studied by Khmaladze (1981), see also 
Doob (1949), which introduces the empirical counterpart of \eqref{b0}.

\begin{remark}
 Introduce $\mu_{\gamma} = \mathcal{L}((\xi_t)_{t \in [0,1]}|\gamma)$ on $C[0,1]$ under the parameter $\gamma$. Then the experiments
 $$E=(Q_1,Q) \sim E'=(\nu_{g_1},\nu_0) \sim (\mu_{\gamma_1},\mu_0) =:F$$ are equivalent in the sense of Le Cam and Yang (2000) when $R(g_1)=\gamma_1$ holds.
\end{remark}
The result follows again from Girsanov's formula
\begin{equation}\label{girsanov}
 \frac{d\mu_{\gamma}}{d\mu_0} = exp \left(\int_0^1 \gamma dB - \frac{\int_0^1 \gamma^2(u)du}2 \right)
\end{equation}
and the isometry property of $R$.\\
\\

In contrast to $Y_t$ the process $(\xi_t)_t$ is a martingale under the change of measure given by $\mu_{- \gamma}$. The equivalence of the different signal models 
$E'$ and $F$ can easily be motivated as follows. Suppose that $g=L(\gamma)$ holds. Then formally \eqref{b0} also holds for the signal alternatives, i.e.
\begin{equation}
 Y_t = B_0(t) + \int_0^t L(\gamma)(u)du = B(t)+ \int_0^t \gamma(u)du - \int_0^t \frac{Y_u}{1-u} du = \xi_t - \int_0^t \frac{Y_u}{1-u} du.
\end{equation}
This decomposition describes the martingale aspects of $Y_t$. We will see that $(\xi_t)_{t \in [0,1]}$ contains all statistical information about
$(Y_t)_{t \in [0,1]}$.\\
\\
The signal model \eqref{signal} and their distributions \eqref{girsanov} now allow for a new interpretation of the experiment $E(t)$ given in Corollary 
\ref{conv2}. We restrict ourselves to $d=1$ asset.

\begin{remark}
 Assume $\gamma_1 = R(g_1)$.\\
(a) We have equivalence of the experiments $E(t) \sim (\mu_{\gamma_1 \ind_{[0,t]}}, \mu_0)$ for each $t \in [0,1]$. Observe that the binary regression model 
$(B_t + \int_0^t \gamma_1(u)du, B_t)$ stands behind this binary experiment.\\
(b) Let $\mathcal{F}_t$ be the canonical filtration given by $(B_t)_{t \in [0,1]}$. Let $\mathbb{X} = (X_t^1)_{t \in [0,1]}$ be the It\^{o} type financial model
with
\begin{equation}\label{ito}
 X_t^1 = exp\left(\int_0^t \gamma_1 dB - \frac{\int_0^t \gamma_1^2(u)du}2 \right) = \frac{d\mu_{\gamma_1|\mathcal{F}_t}}{d\mu_{0|\mathcal{F}_t}}
 =\frac{d\mu_{\gamma_1 \ind_{[0,t]}}}{d\mu_0}
\end{equation} 
corresponding to the experiment of the right hand side of (a). According to Proposition \ref{conv1} the financial model $\mathbb{X}$ is the limit model of 
$\mathbb{X}_n$ of \eqref{price} for $d=1$.
\end{remark}

At present we have weak convergence of the price processes $X_n^1(t)$ only for fixed $t$. On the side of statistical experiments it corresponds to the 
LAN expansion \eqref{lan} of $E_n(t)$ where at time $t$ the parameter $g_1$ is substituted by $E(g_1|\mathcal{F}_{1,t}) = L(\gamma_1 \ind_{[0,t]})$ when 
$g_1 = L(\gamma_1)$ 
holds. The convergence of filtered likelihood processes is given by the so called local asymptotic Wiener property (LAW) which goes back to So and Sen (1981), 
see also Milbrodt (1990). For $d=1$ it reads for our binary experiments as follows.

\begin{definition}\label{lawcond}
 Let $B$ be a standard Brownian motion and let $t \mapsto \sigma^2(t)$ be right continuous non-decreasing on $[0,1]$ with $\sigma^2(0)=0$. The filtered 
 experiments $E_n(t) = (Q_{1,n|\mathcal{F}_{n,t}},Q_{0,n|\mathcal{F}_{n,t}})$ are said to have LAW property when the following conditions hold. There exist 
 stochastic processes $Z_n(t)$ and $\sigma_n^2(t)$ with\\
 (i) $\textnormal{sup}_{t \in [0,1]} \Big| log \frac{dQ_{1,n|\mathcal{F}_{n,t}}}{dQ_{0,n|\mathcal{F}_{n,t}}} 
 - \left[ Z_n(t) - \frac{\sigma_n^2(t)}2\right] \Big| \rightarrow 0$ in $Q_{0,n}$ probability as $n \rightarrow \infty$.\\
 (ii) $Z_n(\cdot) \rightarrow B(\sigma^2(\cdot))$ in $D[0,1]$ under $Q_{0,n}$.\\
 (iii) $\sigma_n^2(\cdot) \rightarrow \sigma^2(\cdot)$ in $D[0,1]$ under $Q_{0,n}$.
\end{definition}

\begin{theorem}\label{law}
 Consider $g_1 = L(\gamma_1)$ for dimension $d=1$. The price process \eqref{price} (filtered likelihood process)
 \begin{equation}
  X_{n,t}^1=log \frac{dQ_{1,n |\mathcal{F}_{n,t}}}{dQ_{0,n |\mathcal{F}_{n,t}}}
 \end{equation}
has LAW property with $\sigma^2(t) = \int_0^t \gamma^2(u)du$ and
\begin{equation}
 Z_n(t,x) = \frac 1{\sqrt n} \sum_{i=1}^n L(\gamma_1 \ind_{[0,t]})(x_i) \rightarrow \int_0^t \gamma dB
\end{equation}
and 
\begin{equation}
 X_{n,t}^1 \rightarrow X_t^1
\end{equation}
for $X_t^1$ of \eqref{ito} both in $D[0,1]$ under $Q_{0,n}$ as $n \rightarrow \infty$.
\end{theorem}
The proof is given in the appendix.\\
\\
This result is of importance for pricing path dependent options. Together with the statistical interpretation by power functions of suitable tests the result 
opens the door for a statistical approach to their option prices.\\
\\
Until now the limit model has deterministic volatility and it is up to a time transformation a financial geometric Brownian motion model. As pointed out earlier 
more general filtered experiments have been studied in the literature. We will summarize some of them.\\
\\
(II) Models with random volatility\\
Below we study so called local asymptotic mixed normal (LAMN) families, see Le Cam and Yang (2000), which naturally show up as hazard based financial models with stochastic volatility. These 
type of limit experiments frequently occur when the parameters are close to the boundary of the parameter space. As in \eqref{likelihood1} and \eqref{pp1} we 
consider again a parameter $g_1 \in \Theta$ with $|g_1| \le K$ given by $L(\gamma_1) = g_1$. Suppose that in addition $Y$ is a $\mathcal{F}_0$-measurable random 
variable which is independent of $B$. Suppose that
$$E(exp(Y^2K^2/2)) < \infty$$
holds. Then
\begin{equation}\label{price2}
 X_t^1 = \frac{d\tilde{Q}_{1|\mathcal{F}_t}}{d\tilde{Q}_{|\mathcal{F}_t}}
 =exp\left(Y \int_0^t \gamma_1 dB - \frac{Y^2\int_0^t \gamma_1^2 du}2\right), \quad t \in [0,1]
\end{equation}
is a financial experiment given by $\tilde{E}:=(\tilde{Q},\tilde{Q}_1)$ with the distribution $\tilde{Q}$ of $B$ and $\tilde{Q}_1$ via \eqref{price2} for $t=1$. 
Observe that $Y$ can be viewed as a random scale parameter which is determined in advance at time $t=0$. Suppose now that the sequence 
$|Y_n| \le \frac{\sqrt n}K$ of random variables is distributional convergent $Y_n \rightarrow Y$ where $Y_n$ is independent of the underlying uniform 
distributions in \eqref{path} - \eqref{lan}. In the model \eqref{pp1} we can now insert $Y_n$ as an additional term, i.e.
\begin{equation}\label{pp2}
 X_{n,t}^1 = \frac{d\tilde{Q}_{1,n |\mathcal{F}_{n,t}}}{d\tilde{Q}_{0,n |\mathcal{F}_{n,t}}}(x) = \prod_{j=1}^n \left(1+\frac{Y_nL(\gamma_1 \ind_{[0,t]})(x_j)}{\sqrt{n}}\right), 
 \quad t \in [0,1],
\end{equation}
where the $x$'s are i.i.d. uniformly distributed under $\tilde{Q}_{0,n}$ as in \eqref{pp1}. We see that \eqref{pp2} arises from the experiment 
$\tilde{E}_n=(\tilde{Q}_{1,n},\tilde{Q}_{0,n})$. The limit experiment of $\tilde{E}_n$ is just $\tilde{E}$ with the limit price process \eqref{price2} of 
\eqref{pp2}. Similarly as in Theorem \ref{law}, a functional limit theorem is obtained for \eqref{pp2} in the sense of Definition \ref{lawcond}. To prove this 
observe that on a new probability space we may assume that $Y_n \rightarrow Y$ holds a.e.. If we then condition under the values of $Y_n$ Theorem \ref{law} 
implies convergence in $D[0,1]$.

\begin{remark}
 It is known from Shiryaev and Spokoiny (2000), sect. 3.7 that stochastic differential equations of the type
 $$dX_t = f_t(X,\theta)dt +dB_t, \quad t \in [0,T]$$
generate a statistical experiment. Under regularity conditions the process $(X_s)_{s \le t}$ corresponds to the filtered distributions
$$\frac{dP_{\theta,t}}{dP_0} = exp\left( \int_0^t f_s(\cdot,\theta)dB - \frac 1 2 \int_0^t f_s(\cdot,\theta)^2ds\right)$$
which fits well into the concept of filtered financial experiments.
\end{remark}
Filtered financial experiments are linked to the so called $(\gamma,\Gamma)$-models of Shiryaev and Spokoiny (2000). They have an interpretation as financial 
model via our duality proposition \ref{equiv}. As illustration consider the following example.

\begin{example}[Shiryaev and Spokoiny (2000), Ex. 3.8]
 Let $M = (M_t)_{t \in [0,T]}$ be a $d$-dimensional square integrable $Q$-martingale w.r.t. a filtration $(\mathcal{F}_t)_{t \in [0,T]}$. Consider given values 
 $x_1,...,x_d \in \mathbb{R}^d$ and assume
 $$E_Q(exp(<x_i,<M>_T x_i>/2) < \infty, \quad i \le d,$$
 where $<M>$ is the quadratic characteristic of $M$. Then
 $$\frac{dQ_i}{dQ}:=exp\left(<x_i,M_T> - \frac{<x_i, <M>_T x_i>}2 \right)$$
 defines a filtered financial experiment $E=(Q_1,...,Q_d,Q)$, $(\mathcal{F}_t)_{t \in [0,T]}$ with the associated price process
 $$X_t^i=\frac{dQ_{i|\mathcal{F}_t}}{dQ_{|\mathcal{F}_t}}:=exp\left(<x_i,M_t> - \frac{<x_i, <M>_t x_i>}2 \right).$$
\end{example}

\section{Beyond arbitrage free financial models}\label{section4}

One might have the impression that the link between the likelihood and financial models only exists when martingale measures are present. The following example 
illustrates that more general examples exist. During the last ten years various authors used a fractional Brownian motion $B_H$ as driving stochastic process in 
finance, see for instance Sottinen and Valkeila (2003) and Hu and \O{}ksendal (2003) among others. Recall that $B_H$ is a centered Gaussian process on 
$C(\mathbb{R})$ with covariance
\begin{equation}
 R(s,t) = \frac 1 2 \left(|t|^{2H} + |s|^{2H} - |t-s|^{2H} \right), \quad s,t \in \mathbb{R},
\end{equation}
where $H \in (0,1)$ is the Hurst index. It was introduced by Kolmogorov (1940). Throughout we will concentrate on the index $1/2 < H < 1$. It is well known that 
$B_H(t)-\frac 1 2 |t|^{2H}$ is the log-likelihood of an important statistical experiment 
$E_H:=(C(\mathbb{R}), \mathcal{B}(C(\mathbb{R})),$ $(Q_t)_{t \in \mathbb{R}})$ given by
\begin{equation}\label{fractional}
 X_t^1 := \frac{dQ_t}{dQ_0}:=exp\left( B_H(t)-\frac 1 2 |t|^{2H}\right).
\end{equation}
Let $B_H$ be a fractional Brownian motion under $Q_0$. It is remarkable that $B_H$ is no semimartingale and martingale arguments do not work when dealing with 
$(X_t^1)_{t \in [0,T]}$. However, $E_H$ is an important Gaussian limit experiment, see Janssen, Milbrodt and Strasser (1985), p. 19ff for an introduction to 
Gaussian experiments. 
Pflug (1982) studied statistical properties of $E_H$, in particular maximum likelihood estimation. It was first pointed out by Prakasa Rao (1968) that $E_H$ is 
the limit experiment of a product of independent regression models
\begin{equation}
 Y_i = \frac t{n^{H-1/2}} + \varepsilon_i, \quad 1 \le i \le n
\end{equation}
when the $\varepsilon_i$ are real i.i.d. random variables with densities $f(x)=C(H) exp\left(-|x|^{2H}\right)$. In other words the product experiment
\begin{equation}
 X_{n,t}^1 = \frac{dP_{t/n^{H-1/2}}^n}{dP_0^n}(x)= \prod_{i=1}^n exp\left( -|x_i - \frac t{n^{H-1/2}}|^{2H} + |x_i|^{2H}\right)
\end{equation}
given by $P_t$ with densities $f(x-t)$ converges to $X_t^1$ given by $E_H$. Thus $(X_{n,t}^1)_{t \in [0,1]}$ has also an interpretation as a sequence of 
financial models with fractional Brownian motion \eqref{fractional} as limit.\\
Notice that nothing is said about martingale measures. More information about the domain of attraction of the statistical experiment $E_H$ can be found in the 
section $A4$ about ``convergence of non-regular experiments to Gaussian experiments'' which is contained in the appendix of Janssen and Mason (1990).

\begin{appendix} 
\section{Appendix}

\textbf{Proof of Theorem \ref{law}:}
All considerations below are carried out for the product of the uniform distribution $Q_{n,0}$. Introduce
$$M_i(t,x_i)=L(\gamma_1 \ind_{[0,t]} | \mathcal{F}_{n,t})(x_i)=E(g_1 |\mathcal{F}_{n,t})(x_i)$$
which are mean zero independent martingales for $i \le n$. The martingale
\begin{equation}
 Z_n(t,x):= \frac 1{\sqrt n} \sum_{i=1}^n M_i(t,x_i)
\end{equation}
has the desired variance $Var(Z_n(t,\cdot))=\sigma^2(t)$ by the isometry $L$.\\
Observe first that condition (ii) of definition \ref{lawcond} follows from the martingale central limit theorem on $D[0,1]$. Here either Rebolledo's martingale 
central limit theorem can be applied or the tightness condition of Loynes (1976) can be verified. The details are easy to prove. We refer to Andersen, Borgan, 
Gill and Keiding (1993), p. 83 or to Rebolledo (1980) for the martingale central limit theorem. Observe that $Z_n(t,\cdot)$ is $Q_{n,0}$ uniformly integrable. 
In order to check condition (ii) we first claim that
\begin{equation}\label{conv3}
 \frac 1 n \sum_{i=1}^n M_i^2(t,x_i) \rightarrow \sigma^2(t)
\end{equation}
is convergent in $D[0,1]$. For this purpose let $M_i(t,\cdot)^2 = \tilde{M}_i(t,\cdot) + A_i(t,\cdot)$ be the Doob-Meyer decomposition of the submartingale 
$M_i^2$. Here $\tilde{M}_i$ are square integrable mean zero martingales and $t \mapsto A_i(t,\cdot)$ is non-decreasing with $E(A_i(t\cdot))=\sigma^2(t)$. The 
proof of \eqref{conv3} splits into two parts. Doob's inequality implies
$$sup_{t \in [0,1]} P\left(\Big|\frac 1 n \sum_{i=1}^n \tilde{M}_i(t,\cdot) \Big| \ge \varepsilon \right) 
\le \frac 1 {\varepsilon^2} \frac{Var(\tilde{M}_1(1,\cdot))}n$$
and
$$sup_{t \in [0,1]} \Big| \frac 1 n \sum_{i=1}^n \tilde{M}_i(t,\cdot) \Big| \rightarrow 0$$
in probability.\\
On the other hand the strong law of large numbers proves
\begin{equation}
 \frac 1 n \sum_{i=1}^n A_i(t,\cdot) \rightarrow \sigma^2(t) \textnormal{ a.e.}
\end{equation}
first for fixed $t$. The same holds a.e. for all $t \in [0,1] \cap \mathbb{Q}$. Since $t \mapsto A_i(t,\cdot)$ is non-decreasing we get convergence a.e. 
uniformly in $t$ by using Polya's theorem and the continuity of $t \mapsto \sigma^2(t)$. Together with the first part \eqref{conv3} follows.\\
The third condition (iii) of the LAW property is based on the Taylor expansion
\begin{equation}
 x - log(1+x) = \frac{x^2}2 (1 + \tilde{R}(x))
\end{equation}
with the remainder term $\tilde{R} \rightarrow 0$ as $t \rightarrow 0$. Define
\begin{equation}
 \frac{\sigma_n^2(t)} 2 = Z_n(t) - log \frac{dQ_{1,n|\mathcal{F}_{n,t}}}{dQ_{0,n|\mathcal{F}_{n,t}}}
 = \sum_{i=1}^n \left[ \frac{M_i(t,x_i)}{\sqrt n} - log\left(1 +\frac{M_i(t,x_i)}{\sqrt n}\right) \right].
\end{equation}
Since $|g| \le K$ is bounded by our assumptions we have $|E(g|\mathcal{F}_{n,t})| \le K$ and thus
$$sup_{t \in [0,1], i \le n} \Big| \tilde{R}\left( \frac{M_i(t,\cdot)}{\sqrt n} \right)\Big| \rightarrow 0.$$
By \eqref{conv3} we have 
\begin{eqnarray*}
 sup_{t \in [0,1]} \Big| \frac{\sigma_n^2(t)} 2 - \frac 1{2n} \sum_{i=1}^n M_i^2(t,x_i) \Big| 
&\le& sup_{t \in [0,1]}\Big|\frac 1{2n} \sum_{i=1}^n M_i^2(t,x_i) \Big| \\
&\cdot& sup_{t \in [0,1], i \le n} \Big| \tilde{R}\left( \frac{M_i(t,\cdot)}{\sqrt n} \right)\Big|
\end{eqnarray*}
with the right hand side converging to $0$ in probability. The arguments finish the proof of the LAW property.
\hspace*{\fill}\begin{small}$\square$\end{small}
\end{appendix}

\bibliographystyle{plain}

\bigskip
\noindent
\parbox[t]{.48\textwidth}{
Arnold Janssen\\
Heinrich-Heine-Universit\"at D\"usseldorf \\
Universit\"atsstr. 1\\
40225 D\"usseldorf, Germany\\
janssena@math.uni-duesseldorf.de } \hfill
\parbox[t]{.48\textwidth}{
Martin Tietje\\
Heinrich-Heine-Universit\"at D\"usseldorf \\
Universit\"atsstr. 1\\
40225 D\"usseldorf, Germany\\
tietje@math.uni-duesseldorf.de }

\end{document}